\documentclass{amsart}

\usepackage{amssymb}
\usepackage[hidelinks]{hyperref}
\usepackage[textsize=footnotesize]{todonotes}
\usepackage{tikz}
\usetikzlibrary{cd}
\usepackage{float}

\theoremstyle{plain}
\newtheorem{theorem}{Theorem}[section]
\newtheorem{lemma}[theorem]{Lemma}
\newtheorem{corollary}[theorem]{Corollary}
\newtheorem{proposition}[theorem]{Proposition}

\theoremstyle{definition}

\newtheorem{question}[theorem]{Question}

\theoremstyle{remark}
\newtheorem{remark}[theorem]{Remark}

\newcommand{\bQ}{\mathbb{Q}}
\newcommand{\Q}{\bQ}

\newcommand{\cA}{\mathcal{A}}
\newcommand{\cB}{\mathcal{B}}
\newcommand{\cC}{\mathcal{C}}

\newcommand{\cD}{\mathcal{D}}
\newcommand{\cE}{\mathcal{E}}
\newcommand{\cF}{\mathcal{F}}

\newcommand{\cI}{\mathcal{I}}
\newcommand{\I}{\cI}
\newcommand{\cJ}{\mathcal{J}}
\newcommand{\J}{\cJ}
\newcommand{\cM}{\mathcal{M}}
\newcommand{\cN}{\mathcal{N}}
\newcommand{\cP}{\mathcal{P}}

\newcommand{\cK}{\mathcal{K}}
\newcommand{\K}{\cK}

\newcommand{\continuum}{\mathfrak{c}}
\newcommand{\cc}{\continuum}
\newcommand{\bnumber}{\mathfrak{b}}
\newcommand{\bb}{\bnumber}
\newcommand{\dnumber}{\mathfrak{d}}
\newcommand{\pnumber}{\mathfrak{p}}

\newcommand{\pp}{\mathfrak{p}}
\newcommand{\dd}{\dnumber}
\newcommand{\cf}{\mathrm{cf}} 
\DeclareMathOperator{\add}{add}
\newcommand{\fin}{\mathrm{Fin}}
\newcommand{\Fin}{\mathrm{Fin}}

\newcommand{\ED}{\mathcal{ED}} 

\newcommand{\CB}{\mathrm{CB}}  
\newcommand{\nwd}{\mathrm{NWD}}  
\newcommand{\mz}{\mathrm{NULL}}  
\newcommand{\conv}{\mathrm{CONV}} 
\newcommand{\ctbl}{\mathrm{CTBL}} 
 
\newcommand{\FIN}{\mathrm{Fin}}

\begin{document}


\title{Two $\mathfrak{b}$ or not two $\mathfrak{b}$?}


\author{Rafa\l{} Filip\'{o}w}
\address[Rafa\l{}~Filip\'{o}w]{Institute of Mathematics\\ Faculty of Mathematics, Physics and Informatics\\ University of Gda\'{n}sk\\ ul.~Wita Stwosza 57\\ 80-308 Gda\'{n}sk\\ Poland}
\email{Rafal.Filipow@ug.edu.pl}
\urladdr{http://mat.ug.edu.pl/~rfilipow}

\author[Adam Kwela]{Adam Kwela}
\address[Adam Kwela]{Institute of Mathematics\\ Faculty of Mathematics\\ Physics and Informatics\\ University of Gda\'{n}sk\\ ul.~Wita  Stwosza 57\\ 80-308 Gda\'{n}sk\\ Poland}
\email{Adam.Kwela@ug.edu.pl}
\urladdr{https://mat.ug.edu.pl/~akwela}


\date{\today}






\begin{abstract}
The paper is devoted to comparison of two generalizations  of the bounding number $\bnumber$.
\end{abstract}


\maketitle


\section{Introduction}

Let $\I$ be an ideal on the set $\omega$ of all nonnegative integers.
The bounding number $\bnumber_\I$ is the least cardinality of an unbounded set in the poset $(\omega^\omega,\leq^\I)$ where $f\leq^\I g$ if $\{n: f(n)> g(n)\}\in \I$.
Then the ordinary bounding number $\bnumber$ is equal to $\bnumber_\fin$ and it is not difficult to show that 
$\bnumber_\I$ is a regular cardinal
and $\bnumber \leq  \bnumber_\I \leq \dnumber$ for every ideal $\I$. In the realm of maximal ideals, the number $\bnumber_\I$ can take various values, for instance there is always a maximal ideal $\I$ with 
$\bnumber_\I=\cf(\dnumber)$ and it is consistent that there exists a maximal ideal $\I$ with $\bnumber_\I=\kappa$ for any regular uncountable $\kappa\leq \continuum$ \cite{Canjar2,Canjar,CanjarPhD}.
However, $\bnumber_\I=\bnumber$ for every ideal $\I$ with the Baire property \cite{MR2537837,MR3624786}  which makes the number $\bnumber_\I$ dull in the realm of definable ideals.

In the literature there are known at least two other ideal-like versions of the bounding number (see below for definitions) which seem more interesting. Below we describe both those versions.
The aim of the paper is to compare these two versions and to calculate values of these numbers for some ideals. Moreover, as a by-product, we answer a question of Kankaanp\"{a}\"{a} \cite{MR3032424}.

Now we start our description of the first version. By an \emph{$\I$-gap} we mean a pair $(\cA,\cB)$ of subfamilies of $\I^+$ such that the families $\cA$ and $\cB$ are \emph{$\I$-orthogonal} (i.e. $A\cap B\in\I$ for all $A\in \cA$ and $B\in \cB$) and there is no $C\subseteq\omega$ which \emph{$\I$-separates} $(\cA,\cB)$ (i.e. there is no $C\subseteq\omega$ with $A\cap C\in \I$ and $B\setminus C\in \I$ for each $A\in \cA$ and $B\in \cB$). If $(\cA,\cB)$ is an $\I$-gap such that $|\cA|=\kappa$ and $|\cB|=\lambda$, then we say that it is an \emph{$\I$-$(\kappa,\lambda)$-gap}.

The study of $\Fin$-gaps has a long history \cite{MR1234288}. For instance, Hausdorff proved that there is a $\fin$-$(\omega_1, \omega_1)$-gap, Rothberger showed that there is a $\Fin$-$(\omega, \bnumber)$-gap and the number $\bnumber$ is the least cardinal $\kappa$ such that there is a $\Fin$-$(\omega, \kappa)$-gap.
On the other hand,  
Todor\v{c}evi\'{c} proved (under PFA) that any $\Fin$-gap is either of type
$(\omega_1 , \omega_1)$ or $(\omega,\bnumber)$.
It is obvious that there is no $\I$-gap for a maximal ideal $\I$.
Recently, there is a growing research on $\I$-gaps for definable ideals $\I$ (e.g. Borel or analytic).  
For instance, Mazur \cite{MR1124539} proved that there is $\I$-$(\omega_1,\omega_1)$-gap for every $F_\sigma$ ideal $\I$, and   Todor\v{c}evi\'{c} \cite{MR1610488} showed that
the $\I$-gaps spectrum 
for $F_\sigma$ and analytic P-ideals $\I$ contains the $\fin$-gap spectrum.

Any $\I$-$(\omega,\kappa)$-gap is called an \emph{$\I$-Rothberger gap} \cite{MR3247032} and the smallest $\kappa$ for which there is an $\I$-Rothberger gap is called  \emph{Rothberger number of $\I$} which we will denote by 
$\bnumber_{R}(\I)$ (we put $\bnumber_{R}(\I)=\cc^+$ in the case there is no $\I$-gap). 
The ordinary bounding number $\bnumber$ is equal to $\bnumber_R(\Fin)$ and $\bnumber_{R}(\I)=\continuum^+$
for maximal ideals $\I$. 
Moreover, it is known \cite[Remark~2.4(6,7)]{MR3247032} that $\bnumber_{R}(\I)=\bnumber$
for analytic P-ideals $\I$
and $\bnumber_{R}(\I)\leq \bnumber$ for $F_\sigma$ ideals $\I$. 
Interestingly,  $\bnumber_R(\I)=\omega_1$
for some  $F_\sigma$ ideals $\I$ \cite{MR3247032}, and consequently 
$\bnumber_R(\I)$ can be consistently strictly smaller than $\bnumber$.

Now we describe the second version of the bounding number that we will be interested in. Let $\geq_\I$ denote the binary relation on $\omega^\omega$ given by 
$$f\geq_\I g \iff  \{n\in\omega: f(n)<g(n)\}\in\I.
$$
Let $\cD_\I$ denote  the family of all functions $f\in\omega^\omega$ such that $f^{-1}(\{n\})\in\I$ for all $n$. Then 
the least cardinality of an unbounded family in the poset $(\cD_\I,\geq_\I\cap (\cD_\I\times\cD_\I))$ will be  denoted by $\bb(\I)$.
The ordinary bounding number $\bnumber$ is equal to $\bnumber(\Fin)$.
In the realm of maximal ideals, the number
$\bnumber(\I)$ can take various values, for instance it is consistent that there exists a maximal ideal $\I$ with $\bnumber(\I) = \kappa$ for any regular
uncountable $\kappa\leq \continuum$ \cite{Canjar2,Canjar,CanjarPhD}.
If $\I$ is a P-ideal, then $\bnumber\leq \bnumber(\I)\leq \dnumber$, and $\bnumber(\I)=\bnumber$ for P-ideals with the Baire property \cite{MR4472525}.
However, there are $F_\sigma$ ideals $\I$ with  $\bnumber(\I)=\omega_1$ \cite{MR4472525}, and consequently $\bnumber(\I)$ can be consistently strictly smaller than $\bnumber$.

Below, we list the main results of this paper.
\begin{itemize}
    \item 
$\bnumber(\I)\leq\bnumber_{R}(\I)$ for every ideal $\I$ (Theorem~\ref{thm:b-vs-b-inequality}),

\item 
It is consistent (Theorems~\ref{thm:bR-graeter-thatn-b} and \ref{thm:b-less-than-bR}) that there are ideals $\I$ with
\begin{itemize}
\item $\bb=\bnumber(\I)<\bnumber_{R}(\I)\leq \continuum$, 
\item $\bb<\bnumber(\I)<\bnumber_{R}(\I)\leq \continuum$,
\item  $\dd<\bb(\I)=\bb_R(\I)\leq\cc$,
\item $\bb<\bnumber(\I)=\bnumber_{R}(\I)\leq \dd=\continuum$.
\end{itemize}

\item 
$\bnumber(\I)=\bnumber_{R}(\I)=\omega_1$ for every ideal $\I$ such that   $\conv\subseteq\I\subseteq\ctbl$ (Theorem~\ref{thm:b-for-conv-and-ctbl}).
In particular, 
	\begin{itemize}
		\item 
	$\bnumber(\conv) = \bnumber_{R}(\conv) = \omega_1$,
\item 	
	$\bnumber(\ctbl) = \bnumber_{R}(\ctbl) = \omega_1$,
	\item 
	$\bnumber(\CB_{\alpha}) = \bnumber_{R}(\CB_{\alpha}) = \omega_1$ for each $\alpha\geq 2$. 
\end{itemize}

\item 
$\bb_R(\I)=\bb(\I)=\bb$
for every P-ideal $\I$ with the Baire property (Theorem~\ref{thm:all-ideal-b-equals-b-for-BP-ideals}).

\item 
$\bnumber_R(\mz)\geq \pnumber$, and consistently 
there is no  $\mz$-$(\omega,\add(\mz))$-gap (Theorem~\ref{thm:b-R-for-NULL}). This answers a question of Kankaanp\"{a}\"{a} \cite{MR3032424}.
\end{itemize}


\section{Preliminaries}

An \emph{ideal} on a set $X$ is a nonempty family $\I\subseteq\cP(X)$ which is closed under taking finite unions (i.e.~if $A,B\in \I$ then $A\cup B\in\I$) and   subsets (i.e.~if $A\subseteq B$ and $B\in\I$ then $A\in\I$), contains all finite subsets of $X$ and $X\notin \I$.
For an ideal $\I$,  we write $\I^+=\{A\subseteq X: A\notin\I\}$ and call it the \emph{coideal of $\I$}.

If $X$ is a countable infinite set,  we consider $2^X=\{0,1\}^X$ as a product (with the product topology) of countably many copies of a discrete topological space $\{0,1\}$ (which is homeomorphic to the Cantor space).
By identifying subsets of $X$  with their characteristic functions,
we equip $\cP(X)$ with the topology of the space $2^X$ and therefore
we can assign topological notions to ideals on $X$.
In particular, an ideal $\I$ is \emph{Borel} (\emph{analytic}, has the \emph{Baire property}, resp.) if $\I$ is a Borel  (analytic, with  the Baire property, resp.) subset of $\cP(X)$.

An ideal $\I$ is a \emph{P-ideal} if  for every sequence $\langle A_n : n\in\omega\rangle$ of elements of $\I$ there is $A\in\I$ such that $A_n\setminus A$ is finite for all $n\in\omega$. 

The vertical section of a set   $A\subseteq X\times Y$ at a point $x\in X$ is defined by $A_{(x)} = \{y\in Y : (x,y)\in A\}$. 
For ideals $\I$ and $\J$ on $X$ and $Y$ respectively, we define the following  ideal (called the \emph{Fubini product} of $\I$ and $\J$):
$$\I\otimes \J = \{A\subseteq X\times Y: \{x\in X:A_{(x)}\notin\J\}\in \I\}.$$
We note that the above construction of the product of ideals also gives an ideal when applied to $\I=\{\emptyset\}$ or $\J=\{\emptyset\}$ which are not ideals according to the definitions (they do not contain all finite sets). Namely, we will consider the following ideals:
\begin{equation*}
    \begin{split}
        \I\otimes \{\emptyset\} 
        &= 
        \{A\subseteq X\times Y: \{x\in X:A_{(x)}\neq \emptyset\}\in \I\},\\
        \{\emptyset\}\otimes \J 
        &= 
        \{A\subseteq X\times Y: A_{(x)}\in \J \text{ for every $x$}\}.
    \end{split}
\end{equation*}



In the studies of $\bnumber_{R}(\I)$, we can restrict to $\I$-gaps satisfying an additional property that  $\cA$ is a partition of $\omega$ \cite[Lemma~2.2]{MR3247032}.
Consequently, a  necessary condition for the number $\bnumber_{R}(\I)$ to be well defined (i.e.~to be $\leq \continuum$) is the existence of an infinite  partition of $\omega$ consisting of $\I$-positive sets. 
On the other hand, such a partition exists if and only if $\I$ is not a direct sum (see page~\pageref{def:direct-sum} for the definition) of finitely many maximal ideals
\cite[Proposition 4.3]{FKL}.
In particular, such a partition exists for every ideal $\I$ with the \emph{hereditary Baire property} (i.e., $\I\restriction X = \{A\cap X: A\in \I\}$ has the Baire property for every $X\in\I^+$) \cite[Proposition 2.5]{FKL}. Note that Borel and analytic ideals have the hereditary Baire property.
The following lemma gives a sufficient condition for $\bnumber_R(\I)$ to be well defined.

\begin{lemma}
\label{lem:when-bR-exists-charaterization}
	 $\bnumber_{R}(\I)\leq \continuum$ if and only if 
	there exists a partition $\{A_n:n<\omega\}$ of $\omega$ such that 
	$A_n\in \I^+$ for each $n$ and 
	$$\forall C\subseteq\omega \left[
	\forall n\, (C\cap A_n\in \I)
	\implies 
	\exists B\in \I^+(B\cap C=\emptyset \land \forall n\, (B\cap A_n\in \I))\right].
	$$
\end{lemma}

\begin{proof}
	($\implies$)
	Let  $(\cA,\cB)$ be a Rothberger $\I$-gap with $|\cB|\leq \continuum$. Without loss of generality, we can assume that $\cA=\{A_n:n<\omega\}$ is a partition of $\omega$ into $\I$-positive sets \cite[Lemma~2.2]{MR3247032}.
	Take any $C\subseteq\omega$ such that $C\cap A_n\in \I$ for each $n$. Since $(\cA,\cB)$ is an $\I$-gap, there is $B\in \cB$ such that $B'=B\setminus C\notin\I$. Then $B'\in\I^+$, $B'\cap C=\emptyset$  and $B'\cap A_n\subseteq B\cap A_n\in \I$ for all $n$.
	
	($\impliedby$)
	Let  $\{A_n:n<\omega\}$ be a partition of $\omega$ into $\I$-positive sets with the required property.
Let  $\cB=\{B\subseteq \omega: B\in \I^+ \text{ and } B\cap A_n\in \I\text{ for all }n\}$.
The pair $(\cA,\cB)$ is $\I$-orthogonal.
Once we show that $(\cA,\cB)$ cannot be $\I$-separated, 
$\bnumber_{R}(\I)\leq |\cB|\leq \continuum$ and the proof will be finished.
	Take $C\subseteq\omega$ such that $C\cap A_n\in \I$ for each $n$. Then there is $B\in\I^+$ such that $B\cap C=\emptyset$ and $B\cap A_n\in \I$ for each $n$.
	Since $B\in \cB$ and $B\setminus C=B\notin\I$, $C$ does not $\I$-separate $(\cA,\cB)$.
\end{proof}

An ideal $\I$ on $X$ is a \emph{$P^+$-ideal} (\emph{$P^+(\I)$-ideal}, resp.)  if for any decreasing sequence $\langle A_n : n\in\omega\rangle$ of $\I$-positive sets there is $A \in \I^+$ such that
$A\setminus A_{n}$ is finite 
($A\setminus A_{n}\in \I$, resp.)
for every $n$.
Every $F_\sigma$-ideal is a $P^+$-ideal \cite[Lemma 3.2.4]{alcantara-phd-thesis}.
Obviously, every $P^+$-ideal is $P^+(\I)$. 
On the other hand, 
the ideal 
$$\Fin^2 = \FIN\otimes \FIN = \{A\subseteq \omega\times \omega: A_{(n)} \text{ is finite for all but finitely many $n$}\}$$ 
is $P^+(\FIN\otimes\FIN)$-ideal but it is not a $P^+$-ideal \cite[Proposition~2.6(2)]{MR4448270}.
On can also show   that the ideal 
$$\conv = \{A\subseteq \Q\cap [0,1]: \text{$A$ has finitely many accumulation points in $[0,1]$}\}$$ 
is $P^+(\conv)$-ideal but it is not a $P^+$-ideal.
Each $P^+(\I)$-ideal has $BW$ property (for definition see \cite{MR2961261}), so 
the following ideals (which do not have $BW$ property \cite{MR2320288}) are not $P^+(\I)$-ideals:
\begin{equation*}
    \begin{split}
\I_d &= \left\{A\subseteq\omega: \limsup_{n\to\infty} \frac{|A\cap n|}{n}=0\right\},
\\
\nwd & = \{A\subseteq\Q\cap [0,1]: \text{$A$ is nowhere dense in $[0,1]$}\}.
    \end{split}
\end{equation*}

The following proposition gives a sufficient condition for $\bnumber_R(\I)$ to be well defined.

\begin{proposition}
Each of the following conditions guarantees  (separately)  that $\bnumber_{R}(\I)\leq \continuum$. 
	\begin{enumerate}
		\item 
$\I$ is a $P^+(\I)$-ideal and there exists an infinite partition of $\omega$ consisting of $\I$-positive sets.

		\item 	
There is $X\in\I^+$ such that $\I\restriction X$ is a $P^+(\I\restriction X)$-ideal and 
there exists an infinite partition of $X$ consisting of $\I$-positive sets.

\item $\I$ is a Borel ideals and 
and there is $X\in\I^+$ such that $\I\restriction X$ is a $P^+(\I\restriction X)$-ideal.
	\end{enumerate}
\end{proposition}

\begin{proof}
	(1) 
    Let $\{A_n:n<\omega\}$ be a partition of $\omega$ such that $A_n\in \I^+$ for each $n$. We will use Lemma~\ref{lem:when-bR-exists-charaterization}. Take any $C\subseteq \omega$ with $C\cap A_n\in \I$ for all $n$.
    Define 
$$D_n=\bigcup \{ A_k\setminus C : k\geq n\}$$
for each $n$.
Since
$D_n\supseteq D_{n+1}$ and $D_n\notin \I$ for all $n$
and $\I$ is a $P^+(\I)$-ideal, there exists $B\in\I^+$ such that $B\setminus D_n\in \I$ for each $n$. In particular, $B'=B\cap D_0\notin\I$. Then $B'\cap A_n \subseteq B\setminus D_{n+1}\in\I$ for each $n$ and $B'\cap C\subseteq D_0\cap C=\emptyset$, so  
$\bnumber_{R}(\I)\leq \continuum$
by Lemma~\ref{lem:when-bR-exists-charaterization}.

	(2) Apply item (1) to the ideal $\I\restriction X$ and use the fact that $\bnumber_{R}(\I)\leq \bnumber_{R}(\I\restriction X)$ \cite[Remark~2.4(2)]{MR3247032}. 

    (3) Follows from item (2) and the fact that Borel ideals have the hereditary Baire property (hence there is an infinite partition of $\omega$ consisting of $\I$-positive sets \cite[Proposition 2.5]{FKL}).
\end{proof}

\begin{question}
Is $\bnumber_{R}(\I)\leq \continuum$ for each Borel ideal $\I$?
\end{question}


\section{\texorpdfstring{$\bnumber(\I)$}{b-I} is less than or equal to  \texorpdfstring{$\bnumber_{R}(\I)$}{b-R-I}}

For an ideal $\I$, we use the following notation:
\begin{itemize}
    \item $\widehat{\cP}_\I$ is the family of all sequences $\langle A_n : n\in\omega\rangle$ such that 
	$A_n\in \I$ for  all $n\in\omega$
	and
	$A_n\cap A_k=\emptyset$ for $n\neq k$;
    \item $\cP_\I$ is the family of all sequences 
	$\langle A_n : n\in\omega\rangle \in \widehat{\cP}_\I$ with 
	$\bigcup\{A_n: n\in \omega\} = \omega$.
\end{itemize}

Using the above notation, we can characterize
the number $\bnumber(\I)$ in the following more technical though useful manner.

\begin{theorem}[{\cite[Theorem~3.10]{MR4472525}}]
\label{thm:b-characterization}
For any ideal $\I$ on $\omega$, 
\begin{equation*}
\begin{split}
\bnumber(\I)  = 
\min \left\{|\cE|:\cE\subseteq\widehat{\cP}_\I 
\text{ and } \right.
& 
\text{for every  $\langle A_n\rangle\in\cP_\I$ there exists $\langle E_n\rangle\in\cE$ with } 
\\ &
\left. 
\text{$\bigcup_{n\in\omega}\left(A_n\cap \bigcup_{i\leq n}E_i\right)\notin\I$}  
\right\}.
\end{split}
\end{equation*}
	
\end{theorem}

With the aid of the above characterization, we can prove the following theorem.

\begin{theorem}
\label{thm:b-vs-b-inequality}
$\bnumber(\I)\leq\bnumber_{R}(\I)$ for every ideal $\I$.
\end{theorem}

\begin{proof}
It is enough to show that $\lambda<\bnumber(\I)$  implies $\lambda<\bnumber_{R}(\I)$.
Let $\lambda<\bnumber(\I)$ and take families $\cA,\cB\subseteq \cP(\omega)$ such that $|\cA|=\omega$, $|\cB|=\lambda$, $\cA=\{A_n:n\in\omega\}$ is a partition of $\omega$ and $A\cap B\in \I$ for any $A\in \cA$, $B\in \cB$.
We have to show that there is $C\subseteq \omega$ such that $A\cap C\in \I$ and $B\setminus C\in \I$ for any $A\in \cA$ and  $B\in \cB$.

For each $A\in \cA$ and $B\in \cB$, we define $E_A^B=A\cap B$. Then $E_A^B\in \I$ for each $A,B$. Moreover, since $\cA$ is a partition of $\omega$, $E_{A_1}^B\cap E_{A_2}^B=\emptyset$ for any distinct $A_1,A_2\in \cA$.

Let $\cE=\{\{E_A^B:A\in \cA\}:B\in \cB\}$.
Since $|\cE|\leq |\cB| = \lambda<\bnumber(\I)$, 
there is a partition $\{C_n:n\in \omega\}$ of $\omega$ such that $C_n\in \I$ for each $n$, and 
$$\bigcup_{n\in \omega}\left(C_n\cap \bigcup_{i\leq n} E^B_{A_i}\right)\in \I$$ for each $B\in \cB$.
We claim that the set 
$$C = \bigcup_{n\in \omega}\left(A_n\cap \bigcup_{i<n} C_i\right)$$
separates the pair $(\cA,\cB)$.
First, we observe that 
$$
A_k\cap C 
= 
A_k\cap \bigcup_{n\in \omega}\left(A_n\cap \bigcup_{i<n} C_i\right)
=
A_k\cap \bigcup_{i<k}C_i
\in \I
$$
for every $k\in\omega$.
Second, we observe that 
\begin{equation*}
\begin{split}
B \setminus C 
& 
= 
B \setminus \bigcup_{n\in \omega}\left(A_n\cap \bigcup_{i<n} C_i\right)
=
\bigcup_{n\in \omega}\left(A_n\cap B\right) \setminus \bigcup_{n\in \omega}\left(A_n\cap \bigcup_{i<n} C_i\right)
\\ & 
=
\bigcup_{n\in \omega}\left(A_n\cap B \setminus  \bigcup_{i<n} C_i\right)
=
\bigcup_{n\in \omega}\left(E^B_{A_n} \setminus  \bigcup_{i<n} C_i\right)
\\ &
\subseteq
\bigcup_{n\in \omega}\left(C_n\cap \bigcup_{i\leq n }E^B_{A_i}\right)
\in\I
\end{split}
\end{equation*}
for every $B\in \cB$. 
\end{proof}

\begin{remark}
In Theorem \ref{thm:bR-characterization} we provide a characterization of $\bnumber_R(\I)$ which is similar to the above characterization of $\bnumber(\I)$, and then Theorem~\ref{thm:b-vs-b-inequality} 
will  easily follow from Theorem~\ref{thm:bR-characterization}, However, we decided to prove Theorem \ref{thm:b-vs-b-inequality} without the use of Theorem \ref{thm:bR-characterization}, since the above proof is much simpler.
\end{remark}

As a collorary of Theorem \ref{thm:bR-characterization},  we can  calculate $\bb(\I)$ for some well-known ideals (for the definitions of ideals mentioned in this theorem see e.g.~\cite{MR3247032}).

\begin{corollary}
	$\bnumber(\I)=\omega_1$ for 
the following ideals 
$\I$: (1)~$\ED$, (2)~$\ED_{\fin}$, 
(3)~fragmented not gradually fragmented ideals $\I = \I_{\{(a_i,\phi_i):j\in\omega\}}$ such that all $\phi_i$ are uniform submeasures
and 
(4)~somewhere tall fragmented ideals 
$\I = \I_{\{(a_i,\phi_i):j\in\omega\}}$ such that all $\phi_i$ are measures.
\end{corollary}

\begin{proof}
Since $\bnumber(\I)\geq \omega_1$ for each $\I$ \cite[Theorem~4.2]{MR4472525}, we only need to show that $\bnumber(\I)\leq \omega_1$ in the case of these ideals. Items (2), (3) and (4) follow from Theorem~\ref{thm:b-vs-b-inequality} and the fact that  $\bnumber_{R}(\I)=\omega_1$ in these cases \cite[Theorem~3.1, 3.4 and 3.6]{MR3247032}. Item (1) follows from 
item (2), the equality  $\ED_{\fin} = \ED\restriction\{(n,k)\in \omega^2: k\leq n\}$
and the inequality $\bnumber(\I)\leq \bnumber(\I\restriction X)$, which holds for every $X\notin \I$ \cite[Theorem~5.1]{MR4472525}.
\end{proof}

\begin{proposition}
\label{prop:b-for-products}
Let  $\I$ and $\J$ be ideals on $\omega$.
\begin{enumerate}
    \item 
$\bnumber_{R}(\I\otimes \J)\leq\bnumber_R(\I)$. 
\item 
If $\bnumber_{R}(\I)  = \bnumber(\I)$, then 
$\bnumber_{R}(\I\otimes \J) = \bnumber(\I\otimes \J)=\bnumber(\I)$.
\end{enumerate}
\end{proposition}

\begin{proof}
(1)
If $\bnumber_{R}(\I)=\continuum^+$, there is nothing to show, so assume that $\bnumber_{R}(\I)\leq \continuum$ and take an $\I$-$(\omega,\bnumber_{R}(\I))$-gap $(\cA,\cB)$.
		Define $\cC = \{A\times\omega: A\in \cA\}$ and $\cD = \{B\times\omega: B\in \cB\}$.
	Once we show that $(\cC,\cD)$ is an $\I\otimes\J$-$(\omega,\bnumber_{R}(\I))$-gap, the proof will be finished.
	Obviously, $|\cC|=\omega$ and $|\cD|=\bnumber_{R}(\I)$. Moreover, since $A\cap B\in \I$, $(A\times \omega)\cap (B\times \omega) = (A\cap B)\times\omega \in \I\otimes\J$. Thus, $(\cC,\cD)$ is $\I\otimes\J$-orthogonal. Suppose that there is a set $C\subseteq \omega\times\omega$ which $\I\otimes\J$-separates the pair $(\cC,\cD)$. Denote $D = \{n\in \omega: C_{(n)}\notin \J\}$, where $C_{(n)} = \{k\in \omega: (n,k)\in C\}$. If we show that the pair $(\cA,\cB)$ is $\I$-separated by the set
    $D$, we will obtain a contradiction.
Take a set $A\in \cA$.
Since $(A\times \omega)\cap C\in \I\otimes\J$, 
$A\cap D \subseteq  \{n\in\omega: ((A\times \omega)\cap C)_{(n)}\notin\J\}\in \I$. Now take a set $B\in \cB$.
Since $(B\times \omega)\setminus  C\in \I\otimes\J$, 
$B\setminus D \subseteq  \{n\in\omega: ((B\times \omega)\setminus  C)_{(n)}\notin\J\}\in \I$. Thus $D$ separates the pair $(\cA,\cB)$.

(2)
It is known that 
$\bnumber(\I) = \bnumber(\I\otimes \J)$ \cite[Theorem 5.13]{MR4472525}, so using item (1) and Theorem~\ref{thm:b-vs-b-inequality}, we obtain 
	$$\bnumber_{R}(\I\otimes \J) \leq  \bnumber_{R}(\I) =\bnumber(\I) = \bnumber(\I\otimes \J) \leq \bnumber_{R}(\I\otimes\J).$$
\end{proof}

For any positive  $n\in \omega$, we define by recursion the ideals $\FIN^n$ by  
$$\FIN^1=\Fin \quad  \text{ and }\quad \FIN^{n+1} = \FIN\otimes \FIN^n.$$

\begin{corollary}
	$\bnumber_{R}(\fin^{n+1}) = \bnumber$ for each $n\in \omega$.
	In particular, $\bnumber_{R}(\fin\otimes\fin) = \bnumber$.
\end{corollary}

\begin{proof}
		Since $\bnumber_{R}(\fin)=\bb(\fin)=\bnumber$ and $\fin^{n+1} = \fin\otimes \fin^{n}$, Proposition~\ref{prop:b-for-products}
finishes the proof.
\end{proof}


\section{Characterization of the Rothberger number}

The  characterization of the number $\bnumber(\I)$ from Theorem~\ref{thm:b-characterization} can be rephrased  in the following manner:
\begin{equation*}
\begin{split}
\bnumber(\I) = 
\min 
\left\{ 
\kappa: 
 \right.
&
\exists_{\{E_n^\alpha: n<\omega \alpha<\kappa\}\subseteq \I} \,
\left[ 
\left( 
\forall_{\alpha<\kappa}\,\forall_{n\in\omega}\, E^\alpha_n\cap \bigcup_{m\neq n} E^\alpha_m=\emptyset
\right) 
\land 
\right. 
\\&
\left.  
\left. 
\left(
\forall_{\{A_n:n<\omega\}\in\cP_\I} \, \exists_{\alpha<\kappa}\, \bigcup_{n\in\omega}
\left(
A_n\cap \bigcup_{i\leq n}E^\alpha_i
\right) 
\notin\I
\right)
\right]
\right\}.
\end{split}
\end{equation*}
In Theorem~\ref{thm:bR-characterization}, we show that the Rothberger number $\bnumber_{R}(\I)$ can be characterized in a similar manner as  the above  characterization of $\bb(\I)$.
It is worth to stress that the \emph{only} difference between the above characterization of
	$\bnumber(\I)$ and the below characterization of $\bnumber_{R}(\I)$ 
	lies in the fact that 
	in the former we only require 
	$$E^\alpha_n\cap \bigcup_{m\neq n} E^\alpha_m=\emptyset$$ 
	for each $\alpha$ separately,
	whereas in the latter we will require disjointness for all $\beta$ simultaneously, i.e.
	$$E^\alpha_n\cap \bigcup_{\beta<\kappa} \left(\bigcup_{m\neq n} E^\beta_m\right)=\emptyset.$$ 

\begin{theorem}
	\label{thm:bR-characterization}
	For any ideal $\I$, 
\begin{equation*}
	\begin{split}
		\bnumber_R(\I) = 
		\min \left\{\kappa:
        \right.
        &
        \left.
\exists_{\{E_n^\alpha: n<\omega,\alpha<\kappa\}\subseteq \I}\,
\left[\left( \forall_{\alpha<\kappa}\, \forall_{n\in\omega}\,E^\alpha_n\cap \bigcup_{\beta<\kappa} \left(\bigcup_{m\neq n} E^\beta_m\right)=\emptyset\right) \land 
		\right.\right.
		\\&
		\left. \left.\left(\forall_{\{A_n:n<\omega\}\in\cP_\I}\, \exists_{\alpha<\kappa}\, \bigcup_{n\in\omega}\left(A_n\cap \bigcup_{i\leq n}E^\alpha_i\right)\notin\I\right)\right]\right\}.
	\end{split}
\end{equation*}
\end{theorem}

\begin{proof}
	$(\geq)$ 
	Let $(\cA,\cB)$ be an $\I$-gap of type $(\omega,\bnumber_{R}(\I))$.
	Without loss of generality, we can assume that $\cA$ is a partition of $\omega$.
	Let $\cA = \{A_n:n\in\omega\}$ and $\cB = \{B_\alpha:\alpha<\bnumber_{R}(\I)\}$.
	We define 
	$$E_n^\alpha = A_n\cap B_\alpha\in\I$$
	for each $n<\omega$ and $\alpha<\bnumber_{R}(\I)$.
	Then $E_n^\alpha\cap E_m^\beta=\emptyset$ for each $n\neq m$ and arbitrary $\alpha,\beta<\bnumber_{R}(\I)$.	
	Take any $\{D_n:n<\omega\}\in\cP_\I$. To finish the proof, we have to show that  
    $$\bigcup_{n\in\omega}\left(D_n\cap \bigcup_{i\leq n}E^\alpha_i\right)\notin\I$$ 
    for some $\alpha<\bnumber_{R}(\I)$.
	Let us define 
	$$C =\bigcup_{n\in\omega} \left(A_n\cap \bigcup_{i< n} D_i\right).$$ 
	Since $C\cap A_n \subseteq \bigcup_{i< n} D_i \in \I$ for each $n$ and $(\cA,\cB)$ is an $\I$-gap, there exists $\alpha<\bnumber_{R}(\I)$ with 
	$B_\alpha\setminus C\notin \I$.
	We claim that 
    $$B_\alpha\setminus C\subseteq \bigcup_{n\in\omega}\left(D_n\cap \bigcup_{i\leq n}E^\alpha_i\right).$$
	Indeed, take $x\in B_\alpha\setminus C$ and let $n\in\omega$ be such that $x\in A_n$. Since $x\notin C$, $x\notin \bigcup_{i< n} D_i$. Then there is $k\geq n$ such that $x\in D_k$.
	Consequently, 
	$$x\in D_k\cap A_n \cap  B_\alpha= D_k\cap E^\alpha_n \subseteq D_k\cap \bigcup_{i\leq k} E^\alpha_i  \subseteq \bigcup_{n\in\omega}\left(D_n\cap \bigcup_{i\leq n}E^\alpha_i\right).$$

	$(\leq)$
	Let $\{E_n^\alpha : n<\omega, \alpha<\kappa\} \subseteq \I$
	be such that 	
	$E_n^\alpha \cap E_m^\beta=\emptyset$ for each $\alpha,\beta<\kappa$ and $n\neq m$
	and for any $\{D_n:n<\omega\} \in \cP_\I$ there is $\alpha<\kappa$ with 
    $$\bigcup_{n\in\omega}\left(D_n\cap \bigcup_{i\leq n}E^\alpha_i\right)\notin\I.$$ 
	Let $\cA=\{A_n:n\in\omega\}$ and $\cB=\{B_\alpha:\alpha<\kappa\}$, where  
	$$A_n =\bigcup_{\alpha<\kappa} E^\alpha_n
	\text{\ \ and\ \ } B_\alpha = \bigcup_{n<\omega} E^\alpha_n$$ 
	for each $n\in\omega$ and $\alpha<\kappa$.
If we show that $(\cA, \cB)$ is an $\I$-$(\omega,\lambda)$-gap  and  $\lambda\leq \kappa$, the proof will be finished.	
	Since $A_n\cap B_\alpha = E^\alpha_n\in \I$, the families $\cA$ and $\cB$ are $\I$-orthogonal.
	Obviously $|\cB|\leq \kappa$. We claim that 
	$|\cA|=\omega$. Indeed, otherwise $A_n=\emptyset$ for all but finitely many $n$, so there is $k$ such that $E^\alpha_n=\emptyset$ for each $\alpha$ and $n\geq k$. Consequently, for any $\{D_n :n<\omega\} \in\cP_\I$ and any $\alpha<\kappa$ we have 
    $$\bigcup_{n\in\omega}\left(D_n\cap \bigcup_{i\leq n} E^\alpha_i\right)\subseteq \bigcup_{n<\omega}E^\alpha_n \in\I,$$ a contradiction.	
	Finally, we show that no $C$ can $\I$-separate $(\cA,\cB)$.	
	Take any $C\subseteq\omega$ such that $A_n\cap C\in \I$ for any $n$.
	We need to find $\alpha<\kappa$ such that $B_\alpha\setminus C\notin\I$.
	We define $D'_n = C\cap A_{n+1}$ for each $n\in \omega$. Let $\omega\setminus \bigcup_{n\in \omega}D'_n = \{a_n:n\in\omega\}$. Now we define $D_n = D'_n\cup\{a_n\}$ for each $n$. Then $\{D_n:n\in \omega\}\in \cP_\I$, so there is $\alpha<\kappa$ 
	with 
    $$\bigcup_{n\in\omega}\left(D_n\cap \bigcup_{i\leq n}E^\alpha_i\right)\notin\I.$$ 
Once we show  that 
$$\bigcup_{n\in\omega}\left(D_n\cap \bigcup_{i\leq n}E^\alpha_i\right) \subseteq (B_\alpha\setminus C) \cup (A_0\cap C),$$    
then 	$B_\alpha\setminus C\notin\I$,
    because
 $A_0\cap C\in \I$.
Since
$$
	\bigcup_{n\in\omega}\left(D_n\cap \bigcup_{i\leq n}E^\alpha_i\right) 
	\subseteq 
	\bigcup_{n\in\omega}\bigcup_{i\leq n}E^\alpha_i
	= B_\alpha,$$
	we have
\begin{equation*}
	\begin{split}
    \bigcup_{n\in\omega}\left(D_n\cap \bigcup_{i\leq n}E^\alpha_i\right) 
 & = 
    \bigcup_{n\in\omega}\left(D_n\cap \bigcup_{i\leq n}E^\alpha_i\right)
    \setminus C 
	\\&
     \cup 
    \bigcup_{n\in\omega}\left(D_n\cap \bigcup_{i\leq n}E^\alpha_i\right)
    \cap  C
	 \\&
    \subseteq 
	(B_\alpha\setminus C)
	\cup
    \left(\bigcup_{n\in\omega}\left(D_n\cap \bigcup_{i\leq n}E^\alpha_i\right)\cap  C\right) 
    .
\end{split}
\end{equation*}
Now we show that 
$$C\cap \bigcup_{n\in\omega}\left(D_n\cap \bigcup_{i\leq n}E^\alpha_i\right)  \subseteq A_0\cap C.$$
	Indeed, 
	using the facts 
	that
    $A_{n+1}\cap \bigcup_{i\leq n}E^\alpha_i=\emptyset$ for each $n$,
    $B_\alpha\cap (\omega\setminus\bigcup_{n\in\omega}A_n)=\emptyset$
    and
$$\{a_n:n\in\omega\} = \left(\omega\setminus \bigcup_{n\in\omega}A_n\right)\cup A_0\cup  \left(\bigcup_{n\geq 1}A_n\setminus C\right),$$
	we obtain 
\begin{equation*}
\begin{split}
\bigcup_{n\in\omega}\left(D_n\cap \bigcup_{i\leq n}E^\alpha_i\right) \cap  C
	& =
\bigcup_{n\in\omega}\left((D'_n\cup \{a_n\})\cap \bigcup_{i\leq n}E^\alpha_i\right) \cap  C
\\&	=
\bigcup_{n\in\omega}\left(((A_{n+1}\cap C) \cup\{a_n\})\cap \bigcup_{i\leq n}E^\alpha_i \right) \cap  C
	\\& =
\bigcup_{n\in\omega}\left(\{a_n\} \cap \bigcup_{i\leq n}E^\alpha_i\right) \cap  C
	\\&
    \subseteq 
	\{a_n:n\in \omega\}\cap B_\alpha \cap C
	\\& =
	\left(\left(\omega\setminus \bigcup_{n\in\omega}A_n\right)\cup A_0\cup  \left(\bigcup_{n\geq 1}A_n\setminus C\right)\right)\cap B_\alpha \cap C 
    \\& = 
    A_0\cap B_\alpha\cap C. 
    	\end{split}
\end{equation*}
\end{proof}


\section{Consistently \texorpdfstring{$\bnumber(\I)$}{b-I} is less than  \texorpdfstring{$\bnumber_{R}(\I)$}{b-R-I}}

We have either $\bnumber_{R}(\I)\leq \bnumber$ or $\bnumber_{R}(\I)=\continuum^+$ for ideals $\I$ for which the value $\bnumber_{R}(\I)$  had been estimated so far in the literature.
Moreover, it is consistent that $\bnumber_{R}(\I)<\bnumber$ for some ideal $\I$  \cite[Theorem~6.1]{MR3247032}.  
In Theorem~\ref{thm:bR-graeter-thatn-b} we show that consistently the value of $\bnumber_{R}(\I)$ can be strictly greater than $\bnumber$ and strictly less than $\continuum^+$. However, this example is cooked up with the aid of maximal ideals, so the following questions remains open.

\begin{question}
Is it consistent that there exists a \emph{Borel} ideal $\I$ such that $\bnumber(\I)<\bnumber_{R}(\I)\leq\cc$?
\end{question}

\begin{lemma}
	\label{lem:b-bR-for-I-times-empty}
	$\bnumber_{R}(\I\otimes\{\emptyset\}) = \bnumber(\I\otimes\{\emptyset\}) = \bnumber(\I)$
    for any ideal $\I$.
\end{lemma}

\begin{proof}
	The second equality is proved in \cite[Theorem~5.13]{MR4472525} and the inequality $\bnumber_{R}(\I\otimes\{\emptyset\}) \geq \bnumber(\I\otimes\{\emptyset\})$ follows from Theorem~\ref{thm:b-vs-b-inequality}.
Below, we show $\bnumber_{R}(\I\otimes\{\emptyset\}) \leq \bnumber(\I)$.

Since, $\bnumber(\I) = \bnumber(\geq_{\I}\cap (\cD_\I\times\cD_\I))$, there exists $\cF\subseteq \cD_\I$ such that $|
\cF|=\bnumber(\I)$ and for each $g\in \cD_\I$ there is $f\in \cF$ such that $f\not\geq_{\I}g$.

Define $A_n=\omega\times\{n\}$ for each $n\in \omega$ and $B_f=\{(n,k)\in \omega^2: f(n)\leq k\}$ for each $f\in \cF$.
Let $\cA = \{A_n:n\in \omega\}$ and $\cB = \{B_f:f\in \cF\}$. Once we show that $(\cA,\cB)$ is an $(\I\otimes\{\emptyset\})$-$(\omega,\bnumber(\I))$-gap, the proof will be finished.

Obviously, $|\cA|=\omega$ and $|\cB|=\bnumber(\I)$. Since $f\in D_\I$, $f^{-1}[\{n\}]\in \I$ for each $n\in \omega$.
Then 
$$A_n\cap B_f \subseteq \bigcup_{k\leq n} f^{-1}[\{k\}]\times \omega \in \I\otimes\{\emptyset\}.$$
Thus,  the pair $(\cA,\cB)$ is $\I\otimes\{\emptyset\}$-orthogonal.

Finally, we have to show that no $C$ can  $\I\otimes\{\emptyset\}$-separate the pair $(\cA,\cB)$.
Let $C\subseteq \omega\times\omega$ be such that $A_n\cap C\in \I\otimes\{\emptyset\}$ for each $n$.
Then all horizontal sections of the set $C$ belongs to $\I$ i.e.~$C^{(n)} = \{k\in \omega: (k,n)\in C\}\in \I$ for each $n$. We define $g:\omega\to\omega$ by 
$$g(i) =
\begin{cases}
i & \text{if }    i\in\omega\setminus\bigcup_{n\in\omega}C^{(n)},\\
n & \text{if } i\in C^{(n)}\setminus \bigcup_{k<n}C^{(k)}.
\end{cases}$$
Since $C^{(n)}\in\I$ for all $n$, $g\in \cD_\I$ and there is $f\in \cF$ such that $f\not\geq_\I g$.
Then $D = \{i\in \omega: f(i)< g(i)\}\notin\I$. Observe that $B_f\setminus C \supseteq \{(i,f(i)): i\in D\}$ as $\{(i,f(i)): i\in D\}\subseteq\{(i,f(i)): i\in \omega\}\subseteq B_f$ and given any $i\in D$ we have two cases:
\begin{itemize}
    \item  $i\in\omega\setminus\bigcup_{n\in\omega}C^{(n)}$, which gives us $C\cap (\{i\}\times\omega)=\emptyset$, so $(i,f(i))\notin C$, or
    \item $i\in C^{(n)}\setminus \bigcup_{k<n}C^{(k)}$ for some $n\in\omega$, which gives us $C\cap (\{i\}\times\omega)\subseteq\{(i,j): j\geq g(i)\}$, so $(i,f(i))\notin C$ by the fact that $i\in D$.
\end{itemize} 
Since $D\notin\I$, we conclude that $B_f\setminus C \supseteq \{(i,f(i)): i\in D\}\notin \I\otimes\{\emptyset\}$, so $C$ does not $\I\otimes\{\emptyset\}$-separate $(\cA,\cB)$.
\end{proof}

\begin{theorem}
	\label{thm:bR-graeter-thatn-b}\
    \begin{enumerate}
        \item If $\dnumber=\continuum$, then there exists an ideal $\I$ with 
$\bnumber(\I) = \bnumber_{R}(\I)=\cf(\dnumber)$. In particular, it is consistent that there exists an ideal $\I$ such that $\bb<\bnumber(\I)=\bnumber_{R}(\I)\leq \dd=\continuum$.
        \item It is consistent that there exists  an ideal $\I$ with $\dd<\bb(\I)=\bb_R(\I)\leq\cc$.
    \end{enumerate}
\end{theorem}

\begin{proof}
	(1)  Canjar \cite{Canjar} proved that, under $\dnumber=\continuum$, there exists a maximal P-ideal $\J$ such that $\bnumber(\J)=\cf(\dnumber)$. Then taking $\I=\J\otimes\{\emptyset\}$ and using Lemma~\ref{lem:b-bR-for-I-times-empty}, we obtain
	$\bb(\I)=\bnumber_{R}(\I)=\cf(\dnumber)$.

    (2) It is  consistent \cite[Theorem 4.2]{more-AK}  (more precisely, in the model obtained by adding $\aleph_2$ random reals to a model of GCH) that there exists an ideal $\J$ with $\dd<\bb(\J)\leq\cc$. Then $\dd<\bnumber_{R}(\J\otimes\{\emptyset\}) = \bnumber(\J\otimes\{\emptyset\}) = \bnumber(\J)\leq\cc$ by Lemma \ref{lem:b-bR-for-I-times-empty}.
\end{proof}

For ideals $\I_i$ on $X_i$ with $i\in \{0,1\}$, we define the following  ideal (called the \emph{direct sum} of $\I_0$ and $\I_1$):\label{def:direct-sum}
$$\I_0\oplus \I_1 = \{A\subseteq (X_0\times\{0\})\cup  (X_1\times\{1\}): \{x\in X_i : (x,i)\in A\} \in\I_i \text{ for $i\in \{0,1\}$}\}.$$

\begin{lemma}
\label{lem:bR-for-direct-sum}
	$\bnumber_{R}(\I\oplus\J) = \min\{\bnumber_{R}(\I),\bnumber_{R}(\J)\}$ for any ideals $\I$ and $\J$.
\end{lemma}

\begin{proof}
	$(\leq)$ 
	If $(\cA,\cB)$ is an $\I$-$(\omega,\bnumber_R(\I))$-gap, then it is not difficult to check that $(\{A\times\{0\}:A\in \cA\}, \{B\times\{0\}: B\in \cB\})$ is an $\I\oplus \J$-$(\omega,\bnumber_R(\I))$-gap. Similarly for $\J$-gaps. 

$(\geq)$
Let $(\cA,\cB)$ be an $\I\oplus\J$-$(\omega,\bnumber_R(\I\oplus\J))$-gap.
For each $A\in \cA$ and $B\in \cB$  let $A^0,A^1, B^0,B^1\subseteq \omega $ be such that $A = (A^0\times\{0\}) \cup (A^1\times\{1\})$
and
$B = (B^0\times\{0\}) \cup (B^1\times\{1\})$. Let 
$\cA^i = \{A^i:A\in \cA\}$ 
and
$\cB^i = \{B^i:B\in \cB\})$ 
for $i=0,1$. If we show that 
$(\cA^0,\cB^0)$ is an $\I$-$(\omega,|\cB^0|)$-gap 
or
$(\cA^1, \cB^1)$ is an $\J$-$(\omega,|\cB^1|)$-gap, 
the proof will be finished (as $|\cB^0|,|\cB^1|\leq \bnumber_R(\I\oplus\J)$).

Since $(A^0\cap B^0)\times \{0\} \cup (A^{1}\cap B^{1})\times\{1\} = A\cap B \in \I\oplus\J$ for each $A\in \cA$ and $B\in \cB$,
the pair $(\cA^0,\cB^0)$ is $\I$-orthogonal
and
the pair $(\cA^1,\cB^1)$ is $\J$-orthogonal.

Suppose that there exists 
$C^0\subseteq\omega$ 
which $\I$-separates 
$(\cA^0,\cB^0)$ 
and there exists 
$C^1\subseteq\omega$ 
which $\J$-separates 
$(\cA^1, \cB^1)$.
Then it is not difficult to see that $C = (C^0\times\{0\})\cup (C^1\times\{1\})$ is a set which $\I\oplus\J$-separates $(\cA,\cB)$ and yields a contradiction. Thus, either $(\cA^0,\cB^0)$ is an $\I$-gap or $(\cA^1,\cB^)1$ is a $\J$-gap. 

Notice that if $(\cA^0,\cB^0)$ is an $\I$-gap then $\cA^0$ is infinite. Indeed, if $\cA^0$ was finite, then the set $(\omega\setminus \bigcup \cA^0)\times\{0\}$ would separate $(\cA^0,\cB^0)$, a contradiction. Similarly we can show that $\cA^1$ is infinite, provided that $(\cA^1,\cB^1)$ is a $\J$-gap.
\end{proof}

\begin{lemma}
\label{lem:b-less-than-bR}
If  $\J$ and $\K$ are  two maximal ideals such that $\bnumber(\J)\leq\bnumber(\K)$, then 
$$\bb_R(\J\oplus (\K\otimes \{\emptyset\}))=\bb(\K)
\quad
\text{and} \quad 
\bb(\J\oplus (\K\otimes \{\emptyset\}))=\bb(\J).$$ 
\end{lemma}

\begin{proof}
Denote $\I=\J\oplus (\K\otimes \{\emptyset\})$. By \cite[Theorems 5.1 and 5.13]{MR4472525} we have 
\[
\bnumber(\I) 
= 
\min\{\bnumber(\J),\bnumber(\K\otimes \{\emptyset\})\} 
= 
\min\{\bnumber(\J),\bnumber(\K)\} 
= 
\bnumber(\J),\]
and by Lemmas~\ref{lem:b-bR-for-I-times-empty} and \ref{lem:bR-for-direct-sum} we have
\[\bnumber_{R}(\I) 
= 
\min\{\bnumber_{R}(\J),\bnumber_{R}(\K\otimes \{\emptyset\})\} 
= 
\min\{\continuum^+,\bnumber(\K)\} 
= 
\bnumber(\K).\]
\end{proof}

If $\I$ is a maximal ideal, then $\bnumber(\I)\leq \continuum < \continuum^+=\bnumber_{R}(\I)$. In Theorem~\ref{thm:b-less-than-bR} we show that $\bnumber(\I)<\bnumber_{R}(\I)\leq \continuum$ is also consistent.

\begin{theorem}
\label{thm:b-less-than-bR}
In the model obtained by adding $\lambda$ Cohen reals to a model of GCH, for every two regular cardinals $\kappa, \kappa'$ such that $\aleph_1\leq \kappa\leq\kappa'< \lambda=\cc$, there is an ideal $\I$ with $\bb(\I)=\kappa$ and $\bb_R(\I)=\kappa'$. In particular, the following statements are consistent:
\begin{enumerate}
    \item there exists an ideal $\I$ such that $\bb=\bnumber(\I)<\bnumber_{R}(\I)\leq \continuum$;
    \item there exists an ideal $\I$ such that $\bb<\bnumber(\I)<\bnumber_{R}(\I)\leq \continuum$.
\end{enumerate}
\end{theorem}

\begin{proof}
    By \cite{Canjar2} and \cite{CanjarPhD}, in the model obtained by adding $\lambda$ Cohen reals to a model of GCH, for every regular cardinals $\aleph_1\leq \kappa\leq\kappa'< \lambda$ there are two maximal ideals $\J$ and $\K$ such that $\bnumber(\J)=\kappa$ and $\bnumber(\K)=\kappa'$. Then Lemma \ref{lem:b-less-than-bR} finishes the proof.
\end{proof}


\section{Technical generalization of the Rathberger number}


Below we define the number $\bnumber_R(\I,\J,\K)$ which can be seen as a convenient tool for studying $\bnumber_{R}(\I)$. 
For instance, we will use it to calculate $\bnumber_{R}(\I)$ for ideals which are between $\conv$ and $\ctbl$ (Section~\ref{subsec:b-R-for-CONV-CTBL}) and to study $\bnumber_{R}(\I)$ for P-ideals with the Baire property
(Section~\ref{subsec:b-R-for-P-BP}).

Let $\I$, $\J$ and $\K$ be ideals on $\omega$.
By an \emph{$(\I,\J,\K)$-$(\kappa,\lambda)$-gap} we mean 
 a pair $(\cA,\cB)$ such that $|\cA|=\omega$, $|\cB|=\lambda$, the families $\cA$ and $\cB$ are $\K$-orthogonal and 
 there is no $C\subseteq\omega$ which \emph{$(\J,\I)$-separates} $(\cA,\cB)$ i.e.~$A\cap C\in \J$ and $B\setminus C\in \I$ for each $A\in \cA$ and $B\in \cB$.
The smallest $\kappa$ for which there is an $(\I,\J,\K)$-$(\omega,\kappa)$-gap will be denoted by 
$\bnumber_{R}(\I,\J,\K)$ and we put $\bnumber_{R}(\I,\J,\K)=\cc^+$
in the case there is no such  gaps.

The following  lemma shows that $\bnumber_{R}(\I,\J,\K)$ is piecewise monotone: decreasing with respect to the last coordinate and increasing with respect to the first and second one.

\begin{lemma}\
\label{lem:monotonicity-of-b-R}
	\begin{enumerate}
		\item If $\K\subseteq \K'$, then $\bnumber_{R}(\I,\J,\K)\geq \bnumber_{R}(\I,\J,\K')$.
		\item If $\J\subseteq \J'$, then $\bnumber_{R}(\I,\J,\K)\leq \bnumber_{R}(\I,\J',\K)$.
		\item If $\I\subseteq \I'$, then $\bnumber_{R}(\I,\J,\K)\leq \bnumber_{R}(\I',\J,\K)$.
	\end{enumerate}
\end{lemma}

\begin{proof}
    Straightforward.
\end{proof}


\subsection{Rothberger number for ideals which are between \texorpdfstring{$\conv$}{conv} and \texorpdfstring{$\ctbl$}{ctbl}}
\label{subsec:b-R-for-CONV-CTBL}

\begin{lemma}	
	\label{lem:b-R-for-conv-and-ctbl}
	$\bnumber_{R}(\ctbl,\ctbl,\conv) \leq  \omega_1$,
where 
$$\ctbl = \{A\subseteq\Q\cap [0,1]: |\overline{A}|\leq \omega\}$$
and  the closure $\overline{A}$ is taken in the space $[0,1]$.
\end{lemma}

\begin{proof}
Let $\{U_n:n\in\omega\}$ be a countable base for the topology on $[0,1]$.
Inductively, we pick infinite sets 
$$A_n\subseteq \Q\cap \left(U_n \setminus \bigcup_{i<n}\overline{A_i}\right)$$ 
such that $\overline{A_n}$ is a perfect nowhere dense set, $\overline{A_n}\subseteq U_n$ and 
$$\overline{A_n}\cap \bigcup_{i<n}\overline{A_i}=\emptyset.$$

	For all $n\in \omega$ and $\alpha<\omega_1$, 
	we pick convergent sequences $x^{n,\alpha} : \omega\to A_n$ having  pairwise distinct limits $y^{n,\alpha}$.
	Then, we define 
    $$B_\alpha = \bigcup_{n\in \omega} \{x^{n,\alpha}(i): i\in\omega\}$$ for each $\alpha<\omega_1$.
	
	Let $\cA=\{A_n:n\in \omega\}$ and $\cB=\{B_\alpha:\alpha<\omega_1\}$. Once we show that $(\cA,\cB)$ is an $(\ctbl,\ctbl,\conv)$-$(\omega,\omega_1)$-gap, the proof will be finished.
	
	Obviously, $|\cA|=\omega$ and $|\cB|=\omega_1$.
	Since $A_n\cap B_\alpha = \{x^{n,\alpha}(i):i\in\omega\} \in \conv$ for any $n$ and $\alpha$, the families $\cA$ and $\cB$ are $\conv$-orthogonal.

	Let $C\subseteq \Q\cap [0,1]$ be such that $A_n\cap C\in \ctbl$ for each $n$.
	We need to find $\alpha$ with $B_\alpha\setminus  C\notin \ctbl$. For each $n$, since $A_n\cap C\in \ctbl$, there is $\alpha_n<\omega_1$ such that 
	$y^{n,\alpha}\notin \overline{A_n\cap C}$ for each $\alpha>\alpha_n$.
Take any countable ordinal $\alpha > \sup\{\alpha_n:n<\omega\}$.
Then $y^{n,\alpha}\notin \overline{A_n\cap C}$ for each $n$, so $x^{n,\alpha}(i)\notin C$ for all but finitely many $i\in \omega$.
Consequently, $x^{n,\alpha}(i)\in B_\alpha\setminus C$ for all but finitely many $i\in \omega$, hence $y^{n,\alpha}\in \overline{B_\alpha \setminus C}$. Since $y^{n,\alpha}\in U_n$, $\overline{B_\alpha \setminus C} = [0,1]$. Thus, $B_\alpha \setminus C\notin\ctbl$.
\end{proof}

For a countable ordinal $\alpha$, 
we have the following \emph{$\alpha$-th Cantor-Bendixson ideal} \cite{alcantara-phd-thesis}
$$\CB_{\alpha} = \{A\subseteq \Q\cap [0,1]: (\overline{A})^{(\alpha)}=\emptyset\},$$
where
 the closure $\overline{A}$ 
and the $\alpha$-th Cantor-Bendixson derivative $(\overline{A})^{(\alpha)}$  
are taken in the space $[0,1]$ (see e.g.~\cite{MR1321597}).
It is not difficult to see  that 
$\CB_{0} = \{\emptyset\}$, $\CB_{1}=\fin$, $\CB_{2}= \conv$
and 
$$\ctbl = \bigcup_{\alpha<\omega_1}\CB_{\alpha}.$$

\begin{theorem}
	\label{thm:b-for-conv-and-ctbl}
If $\conv\subseteq\I\subseteq\ctbl$, then $\bnumber(\I)=\bnumber_{R}(\I)=\omega_1$.
In particular, 
	\begin{enumerate}
		\item 
	$\bnumber(\conv) = \bnumber_{R}(\conv) = \omega_1$,
\item 	
	$\bnumber(\ctbl) = \bnumber_{R}(\ctbl) = \omega_1$,
	\item 
	$\bnumber(\CB_{\alpha}) = \bnumber_{R}(\CB_{\alpha}) = \omega_1$ for each $\alpha\geq 2$. 
\end{enumerate}
\end{theorem}

\begin{proof}
	Using Theorem~\ref{thm:b-vs-b-inequality}
	and the fact that $\bnumber(\I)\geq \omega_1$ for each $\I$ \cite[Theorem~4.2]{MR4472525}, 
	we obtain 
	$\omega_1\leq \bnumber(\I) \leq \bnumber_{R}(\I)$. Now using Lemmas~\ref{lem:monotonicity-of-b-R} and \ref{lem:b-R-for-conv-and-ctbl}, 
we obtain 
$\bnumber_{R}(\I) = \bnumber_{R}(\I,\I,\I) \leq \bnumber_{R}(\ctbl,\ctbl,\conv)\leq \omega_1$.
\end{proof}


\subsection{Rothberger number for P-ideals}
\label{subsec:b-R-for-P-BP}

It is known \cite[Theorem~2.3(7)]{MR3247032} that $\bnumber_{R}(\I)=\bnumber$ for all analytic P-ideals. Below we extend this result to all P-ideals with the Baire property.

\begin{lemma}
\label{bR-I-F-F-for-BP-ideal}
$\bnumber_{R}(\I,\fin,\fin)=\bnumber$
for every ideal  $\I$ with the Baire property.
\end{lemma}

\begin{proof}
	By Lemma~\ref{lem:monotonicity-of-b-R}, we have $\bnumber =\bnumber_{R}(\fin,\fin,\fin)\leq \bnumber_{R}(\I,\fin,\fin)$. We will show that $\bnumber_{R}(\I,\fin,\fin)\leq\bnumber$ for ideals $\I$ with the Baire property.
    
	Since $\I$ has the Baire property,
	there exists a sequence $m_1<m_2<\dots$ such that 
	if there are infinitely many $k$ with $[m_k,m_{k+1})\subseteq A$, then $A\notin \I$
\cite[Th\'{e}or\`{e}me 21]{MR579439} (see also \cite[Theorem 4.1.2]{MR1350295}).
	
	Take any partition $\{D_n:n\in \omega\}$ of $\omega$ into infinite sets and any unbounded family $\{f_\alpha:\alpha<\bnumber\}$ in $(\omega,\leq^*)$.
	
	Define 
	$$A_n = \bigcup_{k\in D_n}[m_k,m_{k+1}) \ \ \ \text{and} \  \ \  
	B_\alpha = \bigcup_{n\in \omega} \bigcup_{k\in D_n\cap f_\alpha(n)} [m_k,m_{k+1})$$
	for each $n\in \omega$ and $\alpha<\bnumber$.
	
	Let $\cA=\{A_n:n\in \omega\}$ and $\cB=\{B_\alpha:\alpha<\bnumber\}$.
	Once we show that $(\cA,\cB)$ is an $(\I,\fin,\fin)$-$(\omega,\kappa)$-gap  with some $\kappa\leq \bnumber$, the proof will be finished.
	
	Using the fact that $D_n$ are pairwise disjoint, we see that $|\cA|=\omega$.
	Obviously, $|\cB|\leq \bnumber$. Moreover, for any $n$ and $\alpha$, 
    $$A_n\cap B_\alpha = \bigcup_{k\in D_n\cap f_\alpha(n)} [m_k,m_{k+1}) \in \fin.$$ 
    Thus $\cA$ and $\cB$ are $\fin$-orthogonal. 
	
	Finally, take any $C\subseteq\omega$ such that $A_n\cap C\in \fin$ for each $n$. We need to find $\alpha$ such that $B_\alpha\setminus C\notin\I$. Let $f:\omega\to\omega$ be defined by
	$$f(n)= \min\{k\in D_n: (A_n\cap C) \cap [m_k,m_{k+1})=\emptyset\}.$$
	Let $\alpha<\bnumber$ be such that the set $E = \{n\in \omega: f(n)<f_\alpha(n)\}$ is infinite. Then $[m_{f(n)}, m_{f(n)+1}) \subseteq B_\alpha \setminus (A_n\cap C)$ for each $n\in E$.
	Since $[m_{f(n)}, m_{f(n)+1}) \subseteq A_n$, 
	$[m_{f(n)}, m_{f(n)+1}) \subseteq B_\alpha \setminus C$ for each $n\in E$.
	Consequently, $B_\alpha \setminus C \notin \I$.
\end{proof}

\begin{lemma}
\label{lem:inequality_for_P-ideals}
    $\bb_R(\I)\leq\bb(\I,\fin,\fin)$ 
    for each P-ideal $\I$.
\end{lemma}

\begin{proof}
    By Lemma \ref{lem:monotonicity-of-b-R}, $\bb_R(\I)\leq\bb_R(\I,\I,\fin)$. To finish the proof, we will show that $\bb_R(\I,\I,\fin)\leq\bb_R(\I,\fin,\fin)$ for every P-ideal $\I$. 

    Let $(\cA,\cB)$ be an $(\I,\fin,\fin)$-$(\omega,\bnumber_{R}(\I,\fin,\fin))$-gap. Then $\cA$ and $\cB$ are $\fin$-orthogonal. Take any $C\subseteq\omega$ such that $A\cap C\in \I$ for all $A\in \cA$. Since $\I$ is a P-ideal and $\cA$ is countable, there exists $D\in \I$ such that $(A\cap C)\setminus D\in \fin$ for each $A\in \cA$. Let $C' = C\setminus D$. Since $A\cap C' \in\fin$ for each $A\in \cA$, there exists $B\in \cB$ with $B\setminus C'\notin\I$. Then $B\setminus C=(B\setminus C')\setminus D\notin\I$ (as $D\in\I$), so $C$ does not $(\I,\I)$-separate $(\cA,\cB)$.
\end{proof}

\begin{theorem}
\label{thm:all-ideal-b-equals-b-for-BP-ideals}
$\bb_R(\I)=\bb(\I)=\bb$
for each P-ideal $\I$ with the Baire property.
\end{theorem}

\begin{proof}
The inequality $\bb_R(\I)\leq\bb$ follows from Lemmas~\ref{bR-I-F-F-for-BP-ideal} and \ref{lem:inequality_for_P-ideals}. The inequalities $\bb_R(\I)\geq\bb(\I)\geq\bb$ follow from Theorem \ref{thm:b-vs-b-inequality} and the fact that $\bb(\I)\geq\bb$ for P-ideals \cite[Theorem 4.5]{MR4472525}.
\end{proof}


\section{Rothberger number for the ideal \texorpdfstring{$\mz$}{null}}
\label{sec:b-R-for-NULL}

Kankaanp\"{a}\"{a} \cite{MR3032424} proved that
$\bnumber_{R}(\nwd)=\add(\cM)$, 
where $\add(\cM)$ is the least cardinal $\kappa$ such that there exists a family of $\kappa$ meager subsets of the real line which union is not meager.
In particular, there is a $\nwd$-$(\omega, \add(\cM))$-gap. He also asked \cite[Question 4]{MR3032424} whether 
there is a $\mz$-$(\omega, \add(\cN))$-gap, 
where $\add(\cN)$ is the least cardinal $\kappa$ such that there exists a family of $\kappa$ Lebesgue null sets which union is not Lebesgue null,
and
$$\mz  = \{A\subseteq\Q\cap [0,1]: \text{$\overline{A}$ has Lebesgue measure zero}\}$$
where  the closure $\overline{A}$ is taken in the space $[0,1]$.
The following theorem shows that the \emph{pseudointersection number} $\pnumber$ (see e.g.~\cite[Definition~6.2]{MR2768685}) is a lower bound on the 
Rothberger number for the ideal $\mz$ and consequently answers the above question.

\begin{theorem}
\label{thm:b-R-for-NULL}
$\bnumber_R(\mz)\geq \pnumber$.
In particular, it is consistent that there is no $\mz$-$(\omega, \add(\cN))$-gap.
\end{theorem}

\begin{proof}
    It is known \cite[Corollary 3.7(b)]{more-AK} that $\pnumber\leq \bnumber(\mz)$, so Theorem~\ref{thm:b-vs-b-inequality} finishes the proof of the inequality.
It is consistent that $\add(\cN)<\pp$ \cite[Section 11.1]{MR2768685}, so the proof of ``in particular'' part is finished.
\end{proof}


\bibliographystyle{amsplain}
\bibliography{Two-b-or-not-two-b}

\end{document}